\documentclass[11pt]{amsart}
\usepackage[T1]{fontenc}
\usepackage[english]{babel}
\usepackage{kpfonts}

\usepackage{amsmath,amssymb,amsthm,enumerate,xspace}
\usepackage{comment}
\usepackage{accents}
\usepackage[colorlinks=true,citecolor=blue,linkcolor=blue]{hyperref}
\usepackage[all]{xy}
\usepackage{tikz}
\usepackage{mathrsfs}

\numberwithin{equation}{section}

\newtheorem{thm}{Theorem}[section]

\newtheorem{prop}[thm]{Proposition}
\newtheorem{lem}[thm]{Lemma}
\newtheorem{cor}[thm]{Corollary}

\newtheorem{ithm}{Theorem}

\newtheorem*{iprob*}{Problem}

\theoremstyle{definition}

\newtheorem{rem}[thm]{Remark}

\newtheorem{exam}[thm]{Example}

\newcommand{\PP}{\mathbf{P}}
\newcommand{\RR}{\mathbf{R}}

\newcommand{\CC}{\mathbf{C}}

\newcommand{\varSL}{\mathrm{SL}}

\newcommand{\hc}{\mathrm{H}_{\mathrm{c}}}
\newcommand{\hcb}{\mathrm{H}_{\mathrm{cb}}}
\newcommand{\hm}{\mathrm{H}_{\mathrm{m}}}
\newcommand{\hh}{\mathrm{H}}

\newcommand{\bu}{\bullet}

\newcommand{\lol}{L}
\newcommand{\ei}{{{}^I\!E}}
\newcommand{\eii}{{{}^{I\!I}\!E}}
\newcommand{\ee}{E}
\newcommand{\difh}{\vec{d}}
\newcommand{\difv}{d^\uparrow}

\newcommand{\inv}{^{-1}}

\newcommand{\lra}{\longrightarrow}

\newcommand{\wt}{\widetilde}

\DeclareMathOperator{\rank}{rank}

\title[The cohomology of semi-simple Lie groups, viewed from infinity]{The cohomology of semi-simple Lie groups,\\ viewed from infinity}
\date{November 2020}

\author[Nicolas Monod]{Nicolas Monod}
\address{EPFL, Switzerland}
\begin{document}

\begin{abstract}
We prove that the cohomology of semi-simple Lie groups admits boundary values, which are measurable cocycles on the Furstenberg boundary. This generalises known invariants such as the Maslov index on Shilov boundaries, the Euler class on projective space, or the hyperbolic ideal volume on spheres.

In rank one, this leads to an isomorphism between the cohomology of the group and of this boundary model. In higher rank, additional classes appear, which we determine completely.
\end{abstract}
\maketitle


\section{Introduction}

The cohomology of a semi-simple Lie group $G$ can be defined, and computed, by a great variety of different methods. For instance, it is often called \emph{continuous} cohomology and denoted by $\hc^\bu(G)$ when realised by continuous cochains on $G$ or on the associated symmetric space. Continuity can be replaced by smoothness, or relaxed to local integrability, even to just measurability.

The van Est isomorphism identifies $\hc^\bu(G)$ with relative Lie algebra cohomology. This then leads to another identification with the cohomology of a \emph{space}, the so-called compact dual symmetric space. Yet another model is the algebraic cohomology of Wigner~\cite{Wigner_PhD, Wigner70}. We recommend the introductions of~\cite{Austin-Moore} and~\cite{Wagemann-Wockel} for modern overviews of these developments, which started in the 1950s but witness contemporary progress. Classical textbooks are~\cite{Guichardet_coho, Borel-Wallach}.

\medskip
Each model has advantages for different applications. The aim of this article is to introduce another viewpoint, placing ourselves on the \emph{Furstenberg boundary} $G/P$ of $G$. This boundary classifies the topological dynamics of $G$ to some extent. It is a homogeneous projective variety for $G$ which covers all other such varieties, among which familiar Grassmannians supporting classical characteristic classes. When viewed instead as a \emph{measurable} $G$-space, this Furstenberg boundary is also a central tool in rigidity theory ever since Furstenberg and Margulis made striking use of it. In addition, $G/P$ is geometrically the space of Weyl chambers at infinity, and as such describes the generic part of the visual boundary at infinity of the symmetric space of $G$.

Our first result shows that the situation is ideal for rank one groups:

\begin{ithm}\label{thm:rk-one}
Let $G$ be a connected semi-simple Lie group of rank one with finite center. Let $P<G$ be a parabolic subgroup. Then the continuous cohomology $\hc^\bu(G)$ is realised by the complex
$$0 \lra \lol(G/P)^G \lra \lol((G/P)^2)^G  \lra \cdots \lra \lol((G/P)^{n+1})^G \lra \cdots $$
of $G$-invariant measurable function classes on the Furstenberg boundary $G/P$.
\end{ithm}

This statement is fundamentally measurable: already for $G=\varSL_2(\RR)$, the cocycles cannot be made continuous on $G/P$, the projective line.

\begin{rem}\label{rem:warnings}
The analogous result for non-trivial coefficients, even unitary, does not hold. It does not hold either for trivial coefficients when $G$ is replaced by a lattice in $G$. Examples are given below.

This indicates that the theorem does not follow immediately from a ``soft'' argument using appropriate (relatively) injective resolutions, or a Buchsbaum criterion, in contrast to many other equivalent characterisations of $\hc^\bu(G)$.
\end{rem}

\begin{rem}
In the rank one setting of Theorem~\ref{thm:rk-one}, Gromov showed that every cohomology class admits bounded representatives~\cite[\S1.2]{Gromov}. This implies that they can be realised by $L^\infty$-cocycles on $G/P$~\cite{Burger-Monod1}, though it is not known whether that complex computes the cohomology of $G$. Nonetheless, it can be deduced that every cohomology class admits a representative on $G/P$ which satisfies the cocycle equation \emph{everywhere}, not just a.e.\ as function class, see~\cite{Monod_lift}, where we also observe that this fails on certain projective varieties in higher rank.
\end{rem}

Return to a general semi-simple Lie group $G$ and denote by $\hm^\bu(G;P)$ the measurable cohomology on the Furstenberg boundary $G/P$. One reason to expect a relation between $\hm^\bu(G;P)$ and the cohomology of $G$ is that the latter can be realised by invariant differential forms on the symmetric space of $G$: such forms are automatically \emph{harmonic}, and suitable harmonic functions have measurable boundary values on $G/P$ by an appropriate Fatou theorem. Notably the Knapp--Williamson theorem~\cite{Knapp-Williamson} establishes this for all bounded harmonic functions, and the cohomology of semi-simple Lie group is conjectured since the 1970s to admit bounded representatives~\cite{Dupont}.

It turns out, however, that in higher rank the result is more involved: the isomorphism still holds outside of a certain range of values, and in that exceptional range we can determine exactly the additional cohomology.

\begin{ithm}\label{thm:main}
Let $G$ be any connected semi-simple Lie group with finite center and let $P<G$ be a minimal parabolic subgroup. Then the cohomology $\hm^q(G;P)$ defined by the  complex
$$0 \lra \lol(G/P)^G \lra \lol((G/P)^2)^G  \lra \cdots \lra \lol((G/P)^{q+1})^G \lra \cdots $$
of $G$-invariant measurable function classes on the Furstenberg boundary $G/P$ coincides with the continuous cohomology $\hc^q(G)$ of $G$ outside the range $3\leq q \leq \rank(G)+2$.

More precisely, writing $r= \rank(G)$:
\begin{itemize}
\item $\hm^q(G;P) \cong \hc^q(G) \oplus \wedge^{q-1} \RR^r$ for odd $3\leq q \leq r+1$;
\item  $\hm^q(G;P) \cong \hc^q(G) \oplus  \wedge^{q-2} \RR^r$ for even $4 \leq q\leq r+2$;
\item $\hm^q(G;P) \cong  \hc^q(G)$ in all other cases.
\end{itemize}
\end{ithm}

The proof shows that the difference between $\hm^q(G;P)$ and $\hc^q(G)$ is given by a canonical embedding of the even-dimensional cohomology $\hc^\bu(A)$ of a maximal split torus $A$ of $G$, with dimensions shifted according to parity. Since $A$ is isomorphic to $\RR^r$, its cohomology is the exterior power $\wedge^\bu \RR^r$.

\medskip

We see that Theorem~\ref{thm:rk-one} is a particular case of this statement since the exceptional range is empty for $r=1$ (taking the parity of $q$ into account).

\medskip

In order to illustrate the geometric meaning of the classes appearing in the exceptional range for $\hm^\bu(G;P)$, we consider the simplest non-trivial example. Let $G=\varSL_2(\RR) \times \varSL_2(\RR)$, so that $r=2$; then $G/P$ can be identified with the product $X=\PP^1\times \PP^1$ of two projective lines. According to Theorem~\ref{thm:main}, the area form in $\RR^2\wedge \RR^2$ gives rise to a generator for  $\hm^3(G;P)$. That is, we obtain a three-dimensional ``volume'' $\Omega$ on $X$ which is projectively invariant. The proof of Theorem~\ref{thm:main} can be coerced into giving the following explicit formula. Considering four points $a,b,c,d\in X$, assume for definiteness that the coordinates are in $\RR$ with $a_i < b_i < c_i < d_i$ for $i=1,2$. The corresponding ``volume'' is then
\begin{multline*}
\Omega(a,b,c,d) = \\
\log [a_1, b_1; c_1, d_1] \log [b_2, c_2; d_2, a_2] - \log [a_2, b_2; c_2, d_2] \log [b_1, c_1; d_1, a_1] ,
\end{multline*}
wherein $[\ldots]$ denotes the cross-ratio. The fact that $\Omega$ is indeed a cocycle is equivalent to the fact that the bivariate function
$$F(x_1, x_2) = \log x_1 \log (1-x_2) - \log x_2 \log (1-x_1) \kern5mm (0 < x1, x_2 < 1)$$
satisfies Rogers'~\cite{Rogers06} form of the the Spence--Abel functional equation:
$$F\left(y\right) - F\left(x\right) \ =\ F\left(\frac{y-x}{1-x}\right) -F\left(\frac{x}{y}\right) + F\left(\frac{x(1-y)}{y(1-x)}\right)$$
and the symmetry $F(x) = - F(1-x)$, where all operations on $x,y$ are understood coordinatewise. 
By contrast, no non-zero \emph{univariate} measurable function $F$ on $(0,1)$ can satisfy these requirements, as follows from~\cite[\S5]{Burger-MonodERN} (while the modified symmetry $F(x) = \zeta(2) - F(1-x)$ characterises Rogers' dilogarithm).

\begin{rem}
The proof of Theorem~\ref{thm:main} will also show that the comparison map $\hcb^\bu(G) \to \hc^\bu(G)$ from the continuous bounded cohomology is induced by the inclusion maps $L^\infty((G/P)^{n+1})\to \lol((G/P)^{n+1})$. We hope that the theorem will be of some guidance towards showing that this comparison map is an isomorphism, as conjectured  e.g.\ in~\cite[Prob.~A]{MonodICM}, and in particular the boundedness conjectured by Dupont~\cite{Dupont}.
\end{rem}

Here are illustrations for Remark~\ref{rem:warnings}.

\begin{exam}\label{exam:unit}
Let $G=\varSL_2(\RR)$. Then there is an irreducible continuous unitary $G$-representation on a Hilbert space $V$, the representation of smallest positive minimal weight, such that $\hc^1(G, V)$ is non-trivial. A similar statement holds when $G$ is (the connected component of) the isometry group of a real or complex hyperbolic space, see e.g.~\cite{Cherix-Cowling-Jolissaint-Julg-Valette} for a detailed geometric construction.

On the other hand, no non-trivial cocycle can be given by a $G$-invariant element $\omega$ of $\lol((G/P)^2, V)$. Indeed, the corresponding crossed homomorphism $\overline\omega\colon G\to V$ (i.e. the corresponding inhomogeneous $1$-cocycle) would automatically be continuous, see Thm.~3 in~\cite{Moore76}. On the other hand, the relation between $\omega$ and $\overline\omega$ is that
$$\omega(g,h) \ = \ g \overline\omega (g\inv h)$$
holds a.e.; in particular $\overline\omega$ descends to $G/P$ and hence has compact range. But a non-trivial affine isometric action has always unbounded orbits, indeed in the present case it is even known that $\overline\omega$ is proper.
\end{exam}

\begin{exam}\label{exam:lattice}
Let again $G=\varSL_2(\RR)$ and consider the fundamental group $\Gamma$ of a closed hyperbolic surface as a uniform lattice in $G$. Since the first Betti number of the surface is non-zero, $\hh^1(\Gamma)$ does not vanish. On the other hand, every $\Gamma$-invariant measurable function on $(G/P)^2$ is essentially constant (and hence trivial in cohomology). Indeed, this follows from the double ergodicity which goes back already to~\cite{Garnett} in this setting.

We observe that the vector-valued double ergodicity introduced in~\cite{Burger-Monod3} shows that this argument provides also an alternative proof for Example~\ref{exam:unit}.
\end{exam}

\medskip
\noindent
\textbf{On previous work.}
Various specific examples of cocycles have long been known to admit privileged representatives on projective varieties associated to $G$, which are all quotients of the Furstenberg boundary. For instance, in degree two, the Maslov index defined on the Langrangian Grassmannian~\cite[\S{C}]{Barge-Ghys92} or more generally on the Shilov boundary~\cite{Clerc-Orsted01} for higher rank. In degree three, there is the Goncharov cocycle for the Borel class of $\varSL_n(\CC)$, see~\cite{Goncharov93} and~\cite[\S2]{Bucher-Burger-Iozzi18}. Another example, in even degree $n$, is the Euler class defined on the projective space $\PP\RR^n$, see~\cite{Smillie_unpublished, Sullivan76}. The spectral sequence that we shall examine below has been used the special case of $\varSL_2(\CC)$ by Bloch~\cite[\S7.4]{Bloch}. In \emph{bounded} cohomology, cocycles on $G/P$ can be used completely generally because $P$ is amenable, see~\cite{Burger-Monod1} and~\cite{Burger-Monod3}. Back to usual cohomology, the case of $G=\mathrm{SO}_+(1,n)$ was considered in~\cite{Pieters18}. For that case, Theorem~\ref{thm:rk-one} fixes an issue with the dimension shifting method in~\cite{Pieters18}, because the lifting maps used in the proof cannot be chosen equivariant. More conceptually, dimension shifting relies on long exact sequences and effaceability; whilst the latter holds for $\hm^\bu(G;P)$, the former does not. This also accounts to the exceptional classes appearing in Theorem~\ref{thm:main}.

\section{Notation and preliminaries}
\begin{flushright}
\begin{minipage}[t]{0.8\linewidth}\itshape\small
\begin{flushright}
A major portion of the following paper is concerned with\\
laying firmer foundations for the theory of Borel cohomology,\\
in the belief that this will hasten defeat of the enemy.
\end{flushright}
\begin{flushright}
\upshape\small
Arthur Mason DuPre III~\cite{Dupre}
\end{flushright}
\end{minipage}
 \end{flushright}

Given a standard measure space $X$ and a Polish topological vector space $V$, we denote by $\lol(X, V)$ the space of measurable function classes $X\to V$. Thus we only take into account the measure \emph{class} on $X$, which is unique on any homogeneous space for a locally compact group and hence will not be apparent in our notation. The space $\lol(X, V)$ is in turn a Polish topological vector space when endowed with the topology of convergence in measure, see~\cite[\S3]{Moore76}. For trivial coefficients, classical references are~\cite[IV.11]{Dunford-Schwartz_I} and~\cite[\S245]{Fremlin2}.

We shall need the following, established in~\cite{Moore76}, Prop.~9:

\begin{lem}\label{lem:exact}
The functor $\lol(X, \cdot)$ preserves the exactness of sequences of Polish topological vector spaces.\qed
\end{lem}

(We recall that the above statement follows readily from the fact that quotient maps of Polish topological vector spaces admit measurable cross-sections.)

Given a second standard measure space $Y$, an appropriate Fubini theorem (Thm.~1 in~\cite{Moore76}) shows that the obvious map gives a well-defined isomorphism of topological vector spaces:
$$\lol\left(X, \lol(Y, V) \right) \ \cong \  \lol(X\times Y, V)\ \cong \ \lol\left(Y, \lol(X, V) \right) .$$

Let $G$ be a locally compact second countable group. If $X$ is endowed with a non-singular $G$-action and $V$ is a (jointly continuous) $G$-module, then so is $\lol(X, V)$~\cite[\S3]{Moore76}. This is the case more generally given a measurable cocycle from $G\times X$ to the automorphisms of $V$. An example of this situation is when  $L<G$ a closed subgroup and $V$ an $L$-module. We then have such a $G$-module $\lol(G/H, V)$, see Prop.~17 in~\cite{Moore76}.

The measurable cohomology $\hm^\bu(G, V)$ is defined in~\cite{Moore76} using the inhomogeneous standard resolution, where cochains for $\hm^n(G, V)$ are $\lol$ function classes $\overline\omega\colon G^n\to V$. There is a well-known isomorphism with the homogeneous resolution, where cochains are $G$-invariant function classes $\omega\colon G^{n+1}\to V$. (Invariant for the usual action means ``equivariant'' as maps to $V$.) Using a Fubini isomorphism, the correspondance between $\overline\omega$ and $\omega$ is given a.e.\ by
$$\omega(g_0, g_1, \ldots, g_n) \ = \ g_0 \overline\omega (g_0\inv g_1, g_1\inv g_2, \ldots, g_{n-1}\inv g_n).$$
We further recall that the homogeneous differentials are defined by the usual alternating sums
$$ d \omega (g_0, \ldots, g_{n+1}) = \sum_{j=0}^{n+1} (-1)^j \omega (g_0, \ldots, \widehat{g_j}, \ldots, g_{n+1})$$
where $\widehat{g_j}$ signifies that the variable $g_j$ has been omitted.

The measurable cohomology satisfies the following version of the Eckmann--Shapiro lemma, see Thm.~6 in~\cite{Moore76}.

\begin{prop}\label{prop:induction}
Let $G$ be a locally compact second countable group and $L<G$ a closed subgroup. For every Polish $H$-module $V$ and every $n\geq 0$ there is a natural isomorphism
$$\hm^n(H, V) \ \cong \ \hm^n\left( G, \lol(G/H, V)\right).\eqno{\qed}$$
\end{prop}

We need a basic vanishing result for compact groups:

\begin{lem}\label{lem:compact}
Let $K$ be a compact metrisable group, $X$ a standard measure space and endow $\lol(X)$ with the trivial $K$-representation. Then $\hm^q(K, \lol(X))$ vanishes for all $q>0$.
\end{lem}

The issue here is that the coefficient space $\lol(X)$ is \emph{not locally convex}. Therefore, although it is known that $\hm^q$ coincides with the continuous cohomology in that setting (see e.g. Rem.~4.13 in~\cite{Wagemann-Wockel}; I am grateful to Friedrich Wagemann for his explanations in this context), in general continuous cohomology of compact groups is \emph{not} known to vanish for Polish topological vector space modules, compare Question~4.2 in~\cite{Austin-Moore}. Nevertheless the particular case of $\lol(X)$ can be dealt with as follows.

\begin{proof}[Proof of Lemma~\ref{lem:compact}]
The statement is well-known for trivial coefficients in $\RR$, see e.g.~\cite[Thm.~A]{Austin-Moore}. (If one knows already that it is equivalent to work with $\hc^\bu$, then an integration argument applies, as was already known in the first years of the theory~\cite[Thm.~2.8]{Hu52}.) Equivalently, the (augmented) inhomogeneous resolution
$$0 \lra \RR \lra \lol(K) \lra \lol(K^2) \lra \cdots$$
is an exact sequence. It follows, by Lemma~\ref{lem:exact}, that the sequence
$$0 \lra \lol(X) \lra \lol\left(X,\lol(K)\right) \lra \lol\left(X,\lol(K^2)\right)  \lra \cdots$$
is also exact. Now we apply the Fubini isomorphisms
$$\lol\left(X,\lol(K^q)\right) \ \cong \lol\left(K^q, \lol(X) \right)$$
and conclude that the (non-augmented) sequence defining $\hm^q(K, \lol(X))$ is exact at all $q>0$.
\end{proof}

Finally we return to a property of $\lol$ that does not involve groups:

\begin{lem}\label{lem:acyclic}
Given any Polish topological vector space $V$, the homogeneous cochain complex
$$0 \lra V \lra \lol(X, V) \lra  \lol(X^2, V) \lra \cdots$$
is exact.
\end{lem}

\begin{proof}
Contrary to more familiar (locally integrable) cases, we cannot construct a homotopy by integrating over the first variable. An explicit proof would consist in observing that the exactness in degree zero follows from Fubini, while in higher degree a dimension shifting argument reduces it inductively to degree zero for a more complicated $V$.

A lazier proof is as follows. The functor (sequence) that associates to $V$ the cohomology of $0 \to \lol(X, V) \to  \lol(X^2, V) \to \cdots$ shares three properties with the measurable cohomology of the \emph{trivial} group: it is just $V$ in degree zero; it takes short exact sequences to long exact sequences (by repeated applications of Lemma~\ref{lem:exact}); and it is effaceable. It follows (Thm.~2 in~\cite{Moore76}) that this functor is isomorphic to the cohomology of the trivial group, whence the statement. (Of course, the proof of the quoted statement contains a similar inductive shifting as alluded to above).
\end{proof}

\section{The cohomology of the minimal parabolic}

We shall need the following result, which might be known to the experts.

\begin{prop}\label{prop:coh:Borel}
Let $G$ be a connected semi-simple Lie group with finite center. Choose a maximal $\RR$-split torus $A<G$ and a minimal parabolic subgroup $P<G$ containing $A$. Then the restriction map
$$\hc^n(P) \lra \hc^n(A)$$
is an isomorphism for all $n$.
\end{prop}

\begin{proof}
Consider a Langlands decomposition $P=MAN$. Since $G$ has finite center, $M$ is compact and therefore it suffices to show the corresponding statement for the restriction from the Borel subgroup $B=AN$, namely from $\hc^n(B)$ to $\hc^n(A)$.

Note first that $\hc^q(N)$ is finite-dimensional for all $q$, for instance because of the van Est isomorphism, see Cor.~III~\S7.3 in~\cite{Guichardet_coho}. Hence it is Hausdorff, see IX~\S3 in~\cite{Borel-Wallach}. This allows for a version of the Lyndon--Hochschild--Serre spectral sequence, see e.g.\ Thm.~9.1 in\cite{Blanc}. This is a sequence with  second tableau
$$E_2^{p,q} = \hc^p\left(A, \hc^q(N)\right)$$
and abutting to $\hc^\bu(B)$. We shall show that $E_2^{p,q}$ vanishes for all $q>0$; on the other hand $E_2^{p,0}= \hc^p(A)$. Thus at least the (finite!) dimensions involved in the restriction map match. This will finish the proof because restriction is onto anyways: indeed, since $A$ is a semi-direct factor of $B$, the restriction admits the inflation $\hc^n(A)\to \hc^n(B)$ as a right inverse.

We turn to the vanishing of $E_2^{p,q}$ for $q>0$. We claim that the natural representation of $A$ on $\hc^q(N)$ does not contain the trivial $A$-representation when $q> 0$.

This claim is probably well-known from Lie algebra cohomology (translating it with the van Est isomorphism). Indeed, for complex Lie algebras it is an elementary case of Kostant's algebraic Borel--Weil--Bott theorem, namely the case of the trivial weight in Thm.~5.14 of~\cite{Kostant61}. (In Kostant's notation, apply that theorem to $\lambda=0$, $\xi_\sigma=0$ and $\mathfrak u = \mathfrak b$, recalling that $g$ therein is the half-sum of positive roots, which is only fixed by the trivial element $\sigma$ of the Weyl group.) The real version has been studied by \v{S}ilhan~\cite{Silhan}, but since we only need the case of the trivial weight, complexification is not an issue.

In any case, we sketch a proof valid over $\RR$ of this elementary case of Kostant's statement: the $A$-representation on $\hc^\bu(N)$ can be realised by the adjoint representation on the Chevalley--Eilenberg complex, which is the dual of $\wedge^\bu\mathfrak n$. Since $A$ is a split torus, this representation can be diagonalised. Since $\mathfrak n$ is the sum of root spaces of negative roots only, the zero weight does not occur at the cochain level on $\wedge^q \mathfrak n$ in non-zero degree $q$. More precisely, any element in the interior of the positive Weyl chamber of $A$ acts as a strict contraction. It follows that the zero weight cannot occur for non-trivial cocycles on the cohomology level either.

The claim implies the desired vanishing of $E_2^{p,q}$ because any abelian group $A$ satisfies $\hc^\bu(A, V)=0$ when $V$ is a finite-dimensional representation not containing the trivial $A$-representation; this follows e.g. from~III~\S3.1 in~\cite{Guichardet_coho}.
\end{proof}

We record that Proposition~\ref{prop:coh:Borel} also implies the following basic vanishing result (which certainly admits simpler proofs).

\begin{cor}\label{cor:van:Borel}
Let $G$ be a connected semi-simple Lie group with finite center and $P<G$  a minimal parabolic subgroup $P<G$. Then the restriction map
$$\hc^n(G) \lra \hc^n(P)$$
vanishes for all $n>0$.
\end{cor}

\begin{proof}
By Proposition~\ref{prop:coh:Borel}, it suffices to show that the restriction to a maximal $\RR$-split torus $A<P$ vanishes. We can assume $G$ simple upon using the K\"unneth formula (after killing the center). The image of this restriction lies within the part of $\hc^n(A)$ that is invariant under the normaliser of $A$ in $G$ since conjugation acts trivially on the cohomology of the conjugating group (see e.g.~I.7 and~III.3 in~\cite{Guichardet_coho}). In other words, we consider the fixed points of the Weyl group action on $\hc^n(A)$. By the van Est isomorphism, $\hc^\bu(A)$ is the dual of the exterior algebra $\wedge^\bu\mathfrak a$. Since $G$ is simple, the Weyl group representation on $\mathfrak a$ is irreducible, see~\cite[V\S4.7]{Bourbaki_Lie4-6}. Using an argument of Steinberg, this implies that the induced representation on the dual of $\wedge^\bu\mathfrak a$ has trivial fixed points, see Ex.~3 in~\cite[V\S2]{Bourbaki_Lie4-6}.
\end{proof}

\section{Modules with few points at infinity}

Let $G$ be a connected semi-simple Lie group with finite center and $P<G$ a minimal parabolic subgroup. Consider the homogeneous differential
$$d\colon \lol (G/P) \lra \lol\left( (G/P)^2 \right),$$
that is, the map given by $(d f) (x,y) = f(y) - f(x)$.

\begin{prop}\label{prop:twist}
The map
$$\hm^q\left(G, \lol (G/P)  \right) \lra \hm^q\left(G,  \lol\left( (G/P)^2 \right) \right)$$
induced by $d$ is an isomorphism when $q$ is odd and vanishes when $q$ is even.
\end{prop}

For the proof, we should clarify the meaning of the \emph{restriction} $\hc^\bu(P) \to \hc^\bu(A)$ occurring in Proposition~\ref{prop:coh:Borel} now that we are in the measurable context of $\hm^\bu$. (Here $A$ is a maximal $\RR$-split torus contained in $P$.) The terminology comes from the fact that continuous cochains can indeed be restricted to any subgroup. However, this does not make sense for measurable cochains up to null-sets, because subgroups are generally null-sets; this is the case of $A<P$. In general, the restriction is induced by the ``forgetful'' natural transformation (given by inclusion) between the functor of $P$-invariants and the functor of $A$-invariants. Concretely, in the setting of the proposition, the cohomology $\hm^q(P)$ can be realised by the complex of $P$-invariants
$$0 \lra \lol(G)^P \lra \cdots \lra  \lol(G^{q+1})^P \lra \cdots $$
Likewise, for $\hm^q(A)$ we consider the complex $\lol(G^{q+1})^A$. Then the restriction is induced by the \emph{inclusion} maps
$$\lol(G^{q+1})^P \lra \lol(G^{q+1})^A.$$

\begin{proof}[Proof of Proposition~\ref{prop:twist}]
We shall work with an explicit formula for the induction isomorphisms of Proposition~\ref{prop:induction} at the level of cochains. Namely, we first realize the cohomology $\hm^\bu(P)$ by the complex of  $P$-invariant homogeneous maps $\alpha' \in\lol(G^{q+1})^P$ as above. Then the corresponding $G$-equivariant map $\alpha$ in $\lol\left(G^{q+1}, \lol (G/P)  \right)^G$ is given by
$$\alpha(g_0, \ldots , g_q) (g P) \ = \ \alpha'(g\inv g_0, \ldots , g\inv g_q).$$
This establishes a well-defined isomorphism of cochain spaces which commutes with the homogeneous differential on the variables in $G$.

Next, we recall that the Bruhat decomposition shows that there is a $P$-orbit of full measure in $G/P$. Namely, the orbit $Pw_0P$, where $w_0$ is (a representative in $G$ of) the longest element of the Weyl group associated to $A$; see e.g.\ Cor.~1.8 in~\cite[IX\S1]{Helgason01}. Equivalently, there is a $G$-orbit of full measure for the diagonal action on $(G/P)^2$. This orbit can be identified with $G/L$, where $L=P \cap w_0 P w_0\inv$. Since $P$ and $w_0 P w_0\inv$ are opposite, we have $L=M A$, where $M$ is a compact group centralising $A$.

Explicitly, the isomorphism of $G$-modules between $\lol\left( (G/P)^2 \right)$ and $\lol(G/M A)$ maps $f$ to the function $\wt f$ defined by $\wt f (gM A) = f(gP, g w_0 P)$.  We now write again $\omega' \leftrightarrow \omega$ for the bijections on the cochain level that induce the induction isomorphism between $\hm^q(MA)$ and $\hm^q\left(G,  \lol (G/M A) \right)$.

In summary, we have isomorphisms
$$\hm^q\left(G,  \lol\left( (G/P)^2 \right) \right) \ \cong \ \hm^q\left(G, \lol(G/MA)\right)  \ \cong \ \hm^q(MA)   \ \cong \ \hm^q(A)$$
where the last isomorphism follows from the compactness of $M$; note that we can implement this isomorphism by the restriction from $MA$ to $A$.

Finally, it remains to identify the map $\hm^q(P) \to \hm^q(A)$ induced by $d$ under all these isomorphisms. If we still denote it by $d$, then we compute
$$(d \alpha') (g_0, \ldots , g_q) = \alpha' (w_0\inv g_0, \ldots , w_0\inv g_q) - \alpha' (g_0, \ldots , g_q).$$
On the other hand, \emph{right} multiplication of any cochains on $G$ is $G$-homotopic to the identity (this well-known fact can be established e.g.\ as in~\cite[I\S7]{Guichardet_coho}). Therefore, the above map has the same effect on cohomology as
$$\alpha' (w_0\inv g_0 w_0, \ldots , w_0 \inv g_q w_0) - \alpha' (g_0, \ldots , g_q),$$
which amounts to $\mathrm{Ad}_{w_0} \mathrm{res} - \mathrm{res}$, wherein $\mathrm{res}$ denotes the restriction from $P$ to $A$ and $\mathrm{Ad}_{w_0}$ refers to the Weyl group action on $\hm^q(A)$. Since $w_0$ is the longest element, it acts by $-1$ on $A$. In particular, it acts trivially on $\hm^q(A)$ when $q$ is even and by $-1$ when $q$ is odd (e.g.\ because $\hm^q(A)$ is isomorphic to the dual of $\wedge^q\mathfrak a$). Thus, the map in the statement of the proposition vanishes indeed when $q$ is even. When $q$ is odd, we find $-2 \mathrm{res}$, which is an isomorphism by Proposition~\ref{prop:coh:Borel}.
\end{proof}

Since we shall need it again, we single out the following observation from the above proof. Note that now $\alpha$ is not assumed to be a cocycle, although $\omega$ is.

\begin{lem}\label{lem:twist:cochain}
Let $q$ be even and let
$$\omega\colon G^{q+1} \lra  \lol\left( (G/P)^2 \right)$$
be a cocycle in the homogeneous resolution for the right hand side in Proposition~\ref{prop:twist}. Suppose that $\omega=d\alpha$ for some $G$-equivariant map
$$\alpha\colon G^{q+1} \lra  \lol(G/P).$$
Then $\omega$ is trivial in cohomology.
\end{lem}

\begin{proof}
The above calculations show that $\omega = d\alpha$, viewed as an element of $\hm^q(A)$, satisfies $\mathrm{Ad}_{w_0} \omega = -\omega$. Since $w_0$ acts as the identity on cohomology for $q$ even, this shows that $\omega$ is trivial.
\end{proof}

Finally, we record what happens to Proposition~\ref{prop:twist} in the elementary case where we have even less points at infinity:

\begin{prop}\label{prop:rest:G-P}
The map
$$\hm^q(G) \lra \hm^q\left(G,  \lol (G/P)\right)$$
induced by the inclusion of constants $\RR\to\lol(G/P)$ vanishes for all $q>0$.
\end{prop}

\begin{proof}
A much simpler computation than above shows that the induction isomorphism intertwines this map to the restriction from $\hm^q(G)$ to $\hm^q(P)$. The statement now follows from Corollary~\ref{cor:van:Borel}.
\end{proof}

\section{Modules with more points at infinity}

\begin{prop}\label{prop:proper}
Let $G$ be a connected semi-simple Lie group with finite center and $P<G$ a minimal parabolic subgroup. Then
$$\hm^q\left( G, \lol\left((G/P)^{p}\right) \right) \ = \ 0$$
when $p\geq 3$ and $q >0$.
\end{prop}

\begin{proof}
We claim that the stabiliser in $G$ of almost every point in $(G/P)^{p}$ is compact; it suffices to prove the claim for $p=3$. 

\begin{flushright}
\begin{minipage}[t]{0.85\linewidth}\itshape\small
\begin{flushright}
Then shalt thou count to three, no more, no less. Three shall be the number thou shalt count, and the number of the counting shall be three.
\end{flushright}
\begin{flushright}
\upshape\small
Book of Armaments, Chap.~2, verses 9--21
\end{flushright}
\end{minipage}
\end{flushright}

Consider a Langlands decomposition $P=MAN$, where $N$ is the unipotent radical of $P$ and $A$ is a maximal $\RR$-split torus of $A$ contained in $P$. The stabiliser of a point in $(G/P)^3$ is the intersection of three minimal parabolics $P_0, P_1, P_2$. 

As in the proof of Proposition~\ref{prop:twist}, we recall that the Bruhat decomposition gives a $P$-orbit $Pw_0P$ of full measure in $G/P$, where $w_0$ is the longest element of the Weyl group. Moreover, this cell can be written as $N w_0 P$, and actually $N$ parametrizes this cell; see e.g.~\cite[8.45]{Knapp02} or Cor.~1.9 in~\cite[IX\S1]{Helgason01}.

Upon conjugating, we may assume $P_0 = P$ and $P_1= w_0 P w_0\inv$; therefore $P_2$ can be parametrised by $n\in N$ as $P_2 = n P_1 n\inv$. Then $P_0 \cap P_1 = MA$ because they are opposite parabolics, and
%
%
%
$$P_0\cap P_2\ = \ n P_0 n\inv \cap n P_1 n\inv  \ =\ n MA n\inv.$$
Thus $P_0 \cap P_1 \cap P_2 = MA\cap n MA n\inv$. The claim now follows from the fact that, for generic $n\in N$, the intersection $A \cap  n A n\inv$ is trivial. Since $A$ normalizes $N$ with $A\cap N$ trivial, the latter fact reduces to the fact that $n$ has trivial centraliser in $A$ for generic $n$, which is apparent on the root space decomposition.

According to the claim, every $G$-orbit in $(G/P)^{p}$ is of the form $G/K$ for some compact subgroup $K<G$. There are at most countably many conjugacy classes in $G$ of such subgroups. Indeed, by Cartan's fixed point theorem we can assume that $K$ belongs to some fixed maximal compact subgroup $K_0$, and even within a compact Lie group conjugacy has countably many classes, see e.g.\ Cor.~1.7.27 in~\cite{Palais60}. Moreover, the $G$-action on $(G/P)^{p}$ is smooth in the Borel sense since it has locally closed orbits (see e.g.~\cite[\S3]{Zimmer84}). In conclusion, the Glimm--Effros theorem (see Thm.~2.9 in~\cite{Effros65}) implies that there is a measurable $G$-isomorphism between $(G/P)^{p}$ and a disjoint countable union of $G$-spaces of the form $G/K_j \times X_j$, where $K_j<G$ is a compact subgroup and $X_j$ is some measure space with trivial $G$-action.

It follows that $\lol \left((G/P)^{p}\right)$ is the (unrestricted) product over $j$ of the $G$-modules $\lol (G/K_j \times X_j)$. Therefore, to trivialize any class in $\hm^q\left( G, \lol\left((G/P)^{p}\right) \right)$, it suffices to show that each $\hm^q\left(G, \lol (G/K_j \times X_j)\right)$ vanishes, recalling that $q >0$. Applying the induction isomorphism (Proposition~\ref{prop:induction}), this amounts to the vanishing of $\hm^q\left(K_j, \lol(X_j)\right)$, which holds by virtue of Lemma~\ref{lem:compact}.
\end{proof}

\section{Proof of the theorem}

For the remaining of the text, let $G$ be a connected semi-simple Lie group with finite center. Choose a maximal $\RR$-split torus $A<G$ and a minimal parabolic subgroup $P<G$ containing $A$ and consider the Langlands decomposition $P=MAN$, where $N$ is the unipotent radical of $P$.

The strategy is to work with the hypercohomology spectral sequence associated to the (augmented) complex of $G$-modules $\lol\left((G/P)^{q}\right)$. Since we will need explicit computations of higher differentials, we give a complete description from scratch, as follows. Consider the bi-complex given for $p,q\geq 0$ by the $G$-invariants
$$C^{p,q} = \lol\left(G^{p+1} \times (G/P)^{q} \right)^G.$$
The two differential maps are given by the homogeneous differentials on the variables in $G^{p+1}$ and $(G/P)^{q+1}$ respectively, up to a sign convention. To minimize confusion, we denote them by $\difv, \difh$ respectively. Specifically, given an element in $C^{p,q}$, we define its first differential in $C^{p+1,q}$ under the Fubini isomorphism by
$$\difv\colon \lol\left(G^{p+1} , \lol\left( (G/P)^{q} \right)\right)^G \lra \lol\left(G^{p+2} , \lol\left( (G/P)^{q} \right)\right)^G$$
Thus, this is the homogeneous resolution for
$$\hm^p\left(G, \lol\left((G/P)^{q} \right)\right).$$
The second differential $\difh$ is defined analogously on the variables in $G/P$ by considering the homogeneous differential
$$\lol\left((G/P)^{q}, \lol\left( G^{p+1} \right)\right)^G \lra \lol\left((G/P)^{q+1} , \lol\left(  G^{p+1}\right)\right)^G$$
but additionally it is affected with the sign $(-1)^{p+1}$.

A bi-complex is a standard setup for first quadrant spectral sequences; classical references compatible with our notations are~\cite[III.7]{Gelfand-Manin} or~\cite[III.14]{Bott-Tu}. In particular we have two spectral sequences $\ei, \eii$ starting with the second, respectively first, differential defined above.

The first computation is routine:

\begin{prop}
These spectral sequences converge to zero.
\end{prop}

\begin{proof}
The point is that the first spectral sequence collapses immediately (and the limits always coincide since they both compute the cohomology of the total complex). Explicitly, $\ei_1^{p,q}$ is by definition the cohomology of the second differential 
$$C^{p,q-1}\lra C^{p,q} \lra C^{p,q+1}$$
as defined above (with the convention $C^{p,-1}=0$). Using the isomorphism with the \emph{inhomogeneous} model for fixed $p$, this amounts to
$$\lol\left(G^{p}, \lol ((G/P)^{q-1}) \right) \lra \lol\left(G^{p}, \lol ((G/P)^{q}) \right) \lra \lol\left(G^{p} , \lol((G/P)^{q+1} )\right)$$
(with still the \emph{homogeneous} differential for $q$, up to a sign). By Lemma~\ref{lem:acyclic}, the sequence with homogeneous differential
$$0 \lra \RR \lra \lol (G/P) \lra \lol\left((G/P)^{2}\right) \lra \lol\left((G/P)^{3} \right) \lra \cdots$$
is exact. This then implies, by Lemma~\ref{lem:exact}, that $\ei_1^{p,q}$ is trivial for all $q$.
\end{proof}

From now on we only work with the second spectral sequence $\eii$, which we simply denote by $\ee$. Recall that $\ee_1^{p,q}$ is defined by the first differential, but by convention $p$ and $q$ are permuted to that in view of our definitions we have $\ee_1^{p,q} = \hm^q\left(G, \lol\left((G/P)^{p} \right)\right)$. Recall also that the differential $d_r$ on page $r\geq 1$, which defines $\ee_{r+1}$, maps $\ee_r^{p,q}$ to $\ee_r^{p+r,q-r+1}$ and that $d_1$ is induced by $\difh$.

The results of the previous sections give us the following information.

\begin{prop}\label{prop:E2}
\ 
  \begin{enumerate}[(i)]
\item $\ee_1^{p,q}=0$ for all $p\geq 3$ and $q >0$.\label{pt:proper}
\item  $\ee_2^{0,q}= \ee_1^{0,q} = \hm^q(G)$ for all $q >0$.\label{pt:G}
\item $\ee_2^{p,q}=0$ for $p= 1,2$ and all odd $q$.\label{pt:odd}
\item $\ee_2^{1,q}=\ee_1^{1,q}\cong \hm^q(P)$ and $\ee_2^{2,q}=\ee_1^{2,q} \cong \hm^q(A)$ for all even $q$.\label{pt:even}
\item $\ee_2^{p+1,0}$ is the cohomology of the $\difv$-complex $\lol\left((G/P)^{p+1}\right)^G$ for all $p>0$. \label{pt:G/P}
  \end{enumerate}
\end{prop}

\begin{proof}
Point~\eqref{pt:proper} is Proposition~\ref{prop:proper}.

Point~\eqref{pt:G} follows from Proposition~\ref{prop:rest:G-P}.

For point~\eqref{pt:odd}, fix $q$ odd and consider the complex defining $\ee_2^{p,q}$ for various $p$:
$$0 \lra \ee_1^{0,q}  \lra \ee_1^{1,q}  \lra \ee_1^{2,q}  \lra \ee_1^{3,q} \lra \cdots$$
We have already noted that $\ee_1^{0,q}  \to \ee_1^{1,q}$ is the zero map and that $\ee_1^{3,q}$ vanishes. Thus the point is equivalent to saying that $\ee_1^{1,q}  \to \ee_1^{2,q}$ is an isomorphism, which was established in Proposition~\ref{prop:twist}.

Similarly, Proposition~\ref{prop:twist} implies point~\eqref{pt:even} in even degree.

Finally, point~\eqref{pt:G/P} follows from $\ee_1^{p+1,0} = \hm^0\left(G, \lol\left((G/P)^{p+1}\right)\right)$.
\end{proof}

The last ingredient that we need in order to understand the higher differentials is:

\begin{prop}\label{prop:d2}
The map $d_2\colon \ee_2^{0,q+1} \to \ee_2^{2,q}$ vanishes for all even $q$.
\end{prop}

(In the special case $q=0$, our identifications above already imply that both $\ee_2^{0,1}$ and $\ee_2^{2,0}$ vanish.)

\begin{proof}[Proof of Proposition~\ref{prop:d2}]
Let $\eta\in C^{q+1, 0}$ be a cocycle representing an element of $\ee_2^{0,q+1}= \ee_1^{0,q+1} = \hm^{q+1}(G)$. We recall how to obtain an element $\omega\in C^{q, 2}$ representing the image $d_2[\eta]$ in $\ee_2^{2,q}=\ee_1^{2,q}\cong \hm^q(A)$:

The fact that $[\eta]$ is in the kernel of $\ee_1^{0,q+1} \to \ee_1^{1,q+1}$ means that $\difh \eta$ is a coboundary. Thus there is $\alpha\in  C^{q, 1}$ such that $\difv \alpha = \difh \eta$. Consider the element $\omega = \difh\alpha$ of $C^{q, 2}$. This is a cocycle because the bi-complex structure implies
$$\difv\omega \ = \ \difv\difh\alpha \ = \ - \difh\difv\alpha \ =\  - \difh \difh \eta \ = \ 0.$$
Now Lemma~\ref{lem:twist:cochain} shows that $[\omega]$ vanishes indeed in $\ee_1^{2,q}$.
\end{proof}

Summarising Propositions~\ref{prop:E2} and~\ref{prop:d2}, the third page $\ee_3$ vanishes except possibly at the following coordinates:

\begin{itemize}
\item $\ee_3^{0, q} = \hm^q(G)$ for all $q>0$;
\item $\ee_3^{p+1,0}$ is the cohomology of the $\difv$-complex $\lol\left((G/P)^{p+1}\right)^G$ for all $p>0$;
\item $\ee_3^{1,q}\cong \hm^q(P)$ and $\ee_3^{2,q}\cong \hm^q(A)$ for all even $q>0$.
\end{itemize}

\noindent
Since the spectral sequence converges to zero, the most immediate consequence is that $d_3$ establishes an isomorphism between $\hm^2(G)$ and $\hm^2(G;P)$. Further consequences are that for every even $q>0$:

\begin{itemize}
\item $d_{q+1}$ gives embeddings of $\hm^q(P)$ into $\hm^{q+1}(G;P)$ and of $\hm^q(A)$ into $\hm^{q+2}(G;P)$;
\item $d_{q+2}$ yields an isomorphism between $\hm^{q+1}(G)$ and the cokernel of $\hm^q(P)$ in $\hm^{q+1}(G;P)$;
\item $d_{q+3}$ yields an isomorphism between $\hm^{q+2}(G)$ and the cokernel of $\hm^q(A)$ in $\hm^{q+2}(G;P)$.
\end{itemize}

This completes the proof of Theorem~\ref{thm:main} if we recall that $\hm^q(P) \cong\hm^q(A)$ by Proposition~\ref{prop:coh:Borel}. Indeed, since $A\cong \RR^r$, its cohomology is the exterior power $\wedge^\bu \RR^r$ (after identifying it with its dual).


\bibliographystyle{amsalpha}
\bibliography{../BIB/ma_bib}

\newcommand{\etalchar}[1]{$^{#1}$}
\def\cprime{$'$}
\providecommand{\bysame}{\leavevmode\hbox to3em{\hrulefill}\thinspace}
\providecommand{\MR}{\relax\ifhmode\unskip\space\fi MR }
\providecommand{\MRhref}[2]{%
  \href{http://www.ams.org/mathscinet-getitem?mr=#1}{#2}
}
\providecommand{\href}[2]{#2}
\begin{thebibliography}{{Bou}81}

\bibitem[AM13]{Austin-Moore}
Tim Austin and Calvin~Cooper Moore, \emph{Continuity properties of measurable
  group cohomology}, Math. Ann. \textbf{356} (2013), no.~3, 885--937.

\bibitem[BBI18]{Bucher-Burger-Iozzi18}
Michelle Bucher, Marc Burger, and Alessandra Iozzi, \emph{The bounded {B}orel
  class and 3-manifold groups}, Duke Math. J. \textbf{167} (2018), no.~17,
  3129--3169.

\bibitem[BG92]{Barge-Ghys92}
Jean Barge and {\'E}tienne Ghys, \emph{Cocycles d'{E}uler et de {M}aslov},
  Math. Ann. \textbf{294} (1992), no.~2, 235--265.

\bibitem[Bla79]{Blanc}
Philippe Blanc, \emph{Sur la cohomologie continue des groupes localement
  compacts}, Ann. Sci. \'Ecole Norm. Sup. (4) \textbf{12} (1979), no.~2,
  137--168.

\bibitem[Blo00]{Bloch}
Spencer~J. Bloch, \emph{Higher regulators, algebraic ${K}$-theory, and zeta
  functions of elliptic curves \upshape (1978 manuscript)}, American
  Mathematical Society, Providence, RI, 2000.

\bibitem[BM99]{Burger-Monod1}
Marc Burger and Nicolas Monod, \emph{Bounded cohomology of lattices in higher
  rank {L}ie groups}, J. Eur. Math. Soc. (JEMS) \textbf{1} (1999), no.~2,
  199--235.

\bibitem[BM02a]{Burger-Monod3}
\bysame, \emph{Continuous bounded cohomology and applications to rigidity
  theory}, Geom. Funct. Anal. \textbf{12} (2002), no.~2, 219--280.

\bibitem[BM02b]{Burger-MonodERN}
\bysame, \emph{On and around the bounded cohomology of {${\rm SL}\sb 2$}},
  Rigidity in dynamics and geometry (Cambridge, 2000), Springer, Berlin, 2002,
  pp.~19--37.

\bibitem[{Bou}81]{Bourbaki_Lie4-6}
N.~{Bourbaki}, \emph{{Groupes et alg\`ebres de Lie. Chapitres 4, 5 et 6}},
  {Elements de Mathematique. Paris etc.: Masson. 288 p. (1981).}, 1981.

\bibitem[BT82]{Bott-Tu}
Raoul Bott and Loring~W. Tu, \emph{Differential forms in algebraic topology},
  Springer-Verlag, New York, 1982.

\bibitem[BW80]{Borel-Wallach}
Armand Borel and Nolan~Russell Wallach, \emph{Continuous cohomology, discrete
  subgroups, and representations of reductive groups}, Princeton University
  Press, 1980.

\bibitem[CCJ{\etalchar{+}}01]{Cherix-Cowling-Jolissaint-Julg-Valette}
Pierre-Alain Cherix, Michael Cowling, Paul Jolissaint, Pierre Julg, and Alain
  Valette, \emph{Groups with the {H}aagerup property}, Progress in Mathematics,
  vol. 197, Birkh\"auser Verlag, Basel, 2001.

\bibitem[C{\O}01]{Clerc-Orsted01}
Jean-Louis Clerc and Bent {\O}rsted, \emph{The {M}aslov index revisited},
  Transform. Groups \textbf{6} (2001), no.~4, 303--320.

\bibitem[DS58]{Dunford-Schwartz_I}
Nelson Dunford and Jacob~Theodore Schwartz, \emph{Linear {O}perators. {I}.
  {G}eneral {T}heory}, With the assistance of W. G. Bade and R. G. Bartle. Pure
  and Applied Mathematics, Vol. 7, Interscience Publishers Inc., New York,
  1958.

\bibitem[DuP68]{Dupre}
Arthur~Mason DuPre, III, \emph{Real {B}orel cohomology of locally compact
  groups}, Trans. Amer. Math. Soc. \textbf{134} (1968), 239--260.

\bibitem[Dup79]{Dupont}
Johan~L. Dupont, \emph{Bounds for characteristic numbers of flat bundles},
  Algebraic topology, Aarhus 1978 (Proc. Sympos., Univ. Aarhus, 1978),
  Springer, Berlin, 1979, pp.~109--119.

\bibitem[Eff65]{Effros65}
Edward~George Effros, \emph{Transformation groups and {$C^{\ast} $}-algebras},
  Ann. of Math. (2) \textbf{81} (1965), 38--55.

\bibitem[Fre03]{Fremlin2}
David~Heaver Fremlin, \emph{Measure theory. {V}ol. 2}, Torres Fremlin,
  Colchester, 2003, Broad foundations, Corrected second printing of the 2001
  original.

\bibitem[Gar83]{Garnett}
Lucy Garnett, \emph{Foliations, the ergodic theorem and {B}rownian motion}, J.
  Funct. Anal. \textbf{51} (1983), no.~3, 285--311.

\bibitem[GM96]{Gelfand-Manin}
Sergei~I. Gelfand and Yuri~I. Manin, \emph{Methods of homological algebra},
  Springer-Verlag, Berlin, 1996, Translated from the 1988 Russian original.

\bibitem[Gon93]{Goncharov93}
Alexander~B. Goncharov, \emph{Explicit construction of characteristic classes},
  I. {M}. {G}el\cprime fand {S}eminar, Adv. Soviet Math., vol.~16, Amer. Math.
  Soc., Providence, RI, 1993, pp.~169--210.

\bibitem[Gro82]{Gromov}
Micha{\"\i}l Gromov, \emph{Volume and bounded cohomology}, Inst. Hautes
  \'Etudes Sci. Publ. Math. (1982), no.~56, 5--99 (1983).

\bibitem[Gui80]{Guichardet_coho}
Alain Guichardet, \emph{Cohomologie des groupes topologiques et des alg\`ebres
  de {L}ie}, Textes Math\'{e}matiques, vol.~2, CEDIC, Paris, 1980.

\bibitem[Hel01]{Helgason01}
Sigurdur Helgason, \emph{Differential geometry, {L}ie groups, and symmetric
  spaces}, Graduate Studies in Mathematics, vol.~34, American Mathematical
  Society, Providence, RI, 2001, Corrected reprint of the 1978 original.

\bibitem[Hu52]{Hu52}
Sze-tsen Hu, \emph{Cohomology theory in topological groups}, Michigan Math. J.
  \textbf{1} (1952), 11--59.

\bibitem[Kna02]{Knapp02}
Anthony~William Knapp, \emph{Lie groups beyond an introduction}, second ed.,
  Progress in Mathematics, vol. 140, Birkh\"auser Boston Inc., Boston, MA,
  2002.

\bibitem[Kos61]{Kostant61}
Bertram Kostant, \emph{Lie algebra cohomology and the generalized
  {B}orel-{W}eil theorem}, Ann. of Math. (2) \textbf{74} (1961), 329--387.

\bibitem[KW71]{Knapp-Williamson}
Anthony~William Knapp and Richard~Edmund Williamson, \emph{Poisson integrals
  and semisimple groups}, J. Analyse Math. \textbf{24} (1971), 53--76.

\bibitem[Mon06]{MonodICM}
Nicolas Monod, \emph{{An invitation to bounded cohomology}}, {Proceedings of
  the international congress of mathematicians (ICM), Madrid, Spain, August
  22--30, 2006. Volume II: Invited lectures. Z\"urich: European Mathematical
  Society, 1183--1211}, 2006.

\bibitem[Mon15]{Monod_lift}
\bysame, \emph{Equivariant measurable liftings}, Fund. Math. \textbf{230}
  (2015), no.~2, 149--165.

\bibitem[Moo76]{Moore76}
Calvin~Cooper Moore, \emph{Group extensions and cohomology for locally compact
  groups. {III}}, Trans. Amer. Math. Soc. \textbf{221} (1976), no.~1, 1--33.

\bibitem[Pal60]{Palais60}
Richard~Sheldon Palais, \emph{The classification of {$G$}-spaces}, Mem. Amer.
  Math. Soc. No. 36, 1960.

\bibitem[Pie18]{Pieters18}
Hester Pieters, \emph{The boundary model for the continuous cohomology of
  {${\rm Isom}^+(\Bbb H^n)$}}, Groups Geom. Dyn. \textbf{12} (2018), no.~4,
  1239--1263.

\bibitem[Rog07]{Rogers06}
Leonard~James Rogers, \emph{{On function sum theorems connected with the series
  \(\sum_{n=1}^{\infty} \frac {x^n}{n^2}\)}}, {Proc. Lond. Math. Soc. (2)}
  \textbf{4} (1907), 169--189.

\bibitem[Smi]{Smillie_unpublished}
John Smillie, \emph{The {E}uler characteristic of flat bundles}, unpublished
  manuscript.

\bibitem[Sul76]{Sullivan76}
Dennis Sullivan, \emph{A generalization of {M}ilnor's inequality concerning
  affine foliations and affine manifolds}, Comment. Math. Helv. \textbf{51}
  (1976), no.~2, 183--189.

\bibitem[\v{S}04]{Silhan}
Josef \v{S}ilhan, \emph{A real analog of {K}ostant's version of the
  {B}ott-{B}orel-{W}eil theorem}, J. Lie Theory \textbf{14} (2004), no.~2,
  481--499.

\bibitem[Wig70a]{Wigner_PhD}
David~Wheeler Wigner, \emph{Algebraic cohomology of topological groups}, Ph.D.
  thesis, Stanford University, 1970.

\bibitem[Wig70b]{Wigner70}
\bysame, \emph{Algebraic cohomology of topological groups}, Bull. Amer. Math.
  Soc. \textbf{76} (1970), 825--826.

\bibitem[WW15]{Wagemann-Wockel}
Friedrich Wagemann and Christoph Wockel, \emph{A cocycle model for topological
  and {L}ie group cohomology}, Trans. Amer. Math. Soc. \textbf{367} (2015),
  no.~3, 1871--1909.

\bibitem[Zim84]{Zimmer84}
Robert~Jeffrey Zimmer, \emph{Ergodic theory and semisimple groups},
  Birkh\"auser Verlag, Basel, 1984.

\end{thebibliography}

\end{document}